\documentclass[10pt]{amsart}
\usepackage{}
\usepackage{amsfonts}
\usepackage{amssymb,latexsym}
\usepackage{amsmath}
\usepackage{amsthm}
\usepackage{amssymb}
\usepackage{enumerate}
\newtheorem{theorem}{Theorem}[section]
\newtheorem{lemma}[theorem]{Lemma}
\newtheorem{proposition}[theorem]{Proposition}

\theoremstyle{definition}
\newtheorem{definition}[theorem]{Definition}

\theoremstyle{remark}
\newtheorem{remark}[theorem]{Remark}
\numberwithin{equation}{section}

\begin{document}
\setlength{\baselineskip}{1.2\baselineskip}
\title  [SPACETIME CONVEXITY]
{ON THE microscopic spacetime convexity principle For FULLY NONLINEAR PARABOLIC EQUATIONS I: SPACETIME CONVEX SOLUTIONS}
\author{Chuanqiang Chen}
\address{Wu Wen-Tsun Key Laboratory of Mathematics\\
         University of Science and Technology of China\\
         Hefei, 230026, Anhui Province, CHINA}
\email{cqchen@mail.ustc.edu.cn}

\thanks{2000 Mathematics Subject Classification: 35K10, 35B99.}
\thanks{Keywords: spacetime convexity; microscopic convexity principle; constant rank theorem.}
\thanks{Research of the author was supported by Chinese Universities Scientific Fund and Wu
Wen-Tsun Key Laboratory of Mathematics in USTC}
\maketitle

\begin{abstract}
Spacetime convexity is a basic geometric property of the solutions
of parabolic equations. In this paper, we study microscopic  convexity properties of spacetime convex solutions of fully
nonlinear parabolic partial differential equations and give a new simple proof of the constant rank theorem in \cite{CH13}.
\end{abstract}

\section{Introduction}
Spacetime convexity is a basic geometric property of the solutions
of parabolic equations. In \cite{Bo82,Bo96,Bo00}, Borell used the
Brownian motion to study certain spacetime convexities of the
solutions of diffusion equations and the level sets of the solution to a heat equations with
Schr$\ddot{o}$dinger potential. Ishige-Salani introduced some notions of parabolic quasiconcavity in \cite{IS10, IS11} and
parabolic power concavity in \cite{IS14}, which are some kinds of spacetime convexity. In \cite{IS10, IS11, IS14}, they studied the corresponding parabolic
boundary value problems using the convex envelope method, which is a macroscopic method. At the same time, Hu-Ma \cite{HM} established a constant rank theorem for the space-time Hessian of space-time convex solutions to the heat equation, which is the microscopic method. Chen-Hu \cite{CH13} generalized the microscopic spacetime convexity principles to fully nonlinear parabolic equations. In this paper we give a new proof of the constant rank theorem in \cite{CH13}.

The convexity of solutions of partial differential equations has been largely
investigated by means of several different methods, which can be
grouped into two  general families: macroscopic and microscopic
methods.

For the macroscopic convexity argument, Korevaar made breakthroughs
in \cite{Kor83_1, Kor83_2}, where he introduced a concavity maximum
principle for a class of quasilinear elliptic equations. The results of
Korevaar were later improved by Kennington \cite{Kenn85} and  Kawohl \cite{Kaw86}.
The theory was further developed to its great generalization by
Alvarez-Lasry-Lions \cite{All97}. There are some related results on
the spacetime convexity of the solutions of parabolic equations in
\cite{Kenn88, PS} using the similar elliptic macroscopic convexity
technique in Kennington \cite{Kenn85} and Kawohl \cite{Kaw86}.

The key of the
study of microscopic convexity is a method called constant rank
theorem, which is a very useful tool to produce convex solutions in
geometric analysis. The constant rank theorem technique was introduced in
dimension 2 by Caffarelli-Friedman \cite{CF85} (see also the work of Singer-Wong-Yau-Yau \cite{SYY85} for a similar
approach). Later the result in \cite{CF85} was generalized to high dimensions
by Korevaar-Lewis \cite{KL87}. Recently the constant rank theorem
was generalized to fully nonlinear elliptic and parabolic equations
in \cite{CGM07, BG09, BG10}. For parabolic equations, the
constant rank theorems in \cite{CGM07, BG09} are only about the
spatial hessian of the solution. As geometric applications of
the constant rank theorem, the Christoffel-Minkowski problem and the
related prescribing Weingarten curvature problems were studied in
\cite{GM03,GLM06,GMZ06}. The preservation of convexity for the
general geometric flows of hypersurfaces was given in \cite{BG09}.
Soon after, a constant rank theorem for the level sets was
established in \cite{BGMX}, where the result is a microscopic
version of \cite{BLS} (also it was studied in \cite{Kor90}). The
existence of the $k$-convex hypersurfaces with prescribed mean
curvature was given in \cite{HMW09} recently. Related results are also in \cite{LMX10, MX08}.

In this paper, we consider the following fully nonlinear parabolic equation
\begin{equation}\label{1.4}
\frac{{\partial u}} {{\partial t }}=F(D^2 u,Du,u,x,t), \quad (x,t)
\in \Omega \times (0,T] ,
\end{equation}
where $ F=F(A,p,u,x,t) \in C^{2,1} (\mathcal{S}_+^n \times
\mathbb{R}^n \times \mathbb{R} \times \Omega \times [0,T])$, $D^2 u =(\frac{\partial^2 u}{\partial x_i \partial x_j})_{n \times n}$, $Du =(u_{x_1}, \cdots, u_{x_n})$, and $F$
is elliptic in the following sense
\begin{equation}\label{1.5}
(\frac{{\partial F}} {{\partial u_{ij} }})
> 0,  \quad \text{for all } (x,t) \in \Omega \times (0,T].
\end{equation}
where $\mathcal{S}_+^n$ denotes the set of all the semipositive
definite $n \times n$ matrices. We say equation \eqref{1.4} is
parabolic if $F$ is elliptic in the sense of \eqref{1.5}.

In \cite{BG09}, Bian-Guan consider the (spatial) Hessian of the (spatial) convex solutions of \eqref{1.4}
and get the following result.

\begin{theorem} \label{th1.1}
Suppose $ F(A,p,u,x,t) \in C^{2,1} (\mathcal{S}_+^n \times \mathbb{R}^n \times \mathbb{R}
\times \Omega \times [0,T))$. Assume $F$ satisfies condition \eqref{1.5} and
\begin{equation}\label{1.6}
F(A^{-1},p,u,x,t)  \text{ is locally convex in } (A,u,x) \text{ for each pair } (p,t).
\end{equation}
Let $u \in C^{2,1} (\Omega \times [0,T)) $ be a convex solution of
the equation
\begin{equation}
\frac{{\partial u}} {{\partial t}} = F(D^2 u,Du,u,x,t). \notag
\end{equation}
For each $t \in (0, T)$, let $l(t)$ be the minimal rank of $(D^2 u(x, t))$ in $\Omega$, then the rank of
$(D^2u(x, t))$ is constant $l(t)$ and $l(s)\leq l(t)$ for all $s \leq t < T$. For each $0 < t \leq T$, $x_0 \in \Omega$
there exist a neighborhood $\mathcal {U}$ of $x_0$ and $(n-l(t))$ fixed directions $V_1, \cdots, V_{n-l (t)}$ such that
$D^2u(x, t)V_j = 0$ for all $1 \leq j \leq n - l(t)$ and $x \in \mathcal {U}$. Furthermore, for any $t_0 \in [0,T )$,
there is a $\delta > 0$, such that the null space of $(D^2u(x, t))$ is parallel in $(x, t)$ for all
$x \in \Omega$, $t \in (t_0, t_0 +\delta)$.
\end{theorem}

Naturally, one can consider the spacetime Hessian of the spacetime convex solutions
of \eqref{1.4} and establish the corresponding constant rank theorem. Hu-Ma \cite{HM} established the
spacetime constant rank theorem for the heat equation in $\mathbb{R}^n $, and Chen-Hu \cite{CH13} extended
it to fully nonlinear parabolic equations with the "inverse convex" structural condition. But the proof of the spacetime constant rank
theorem is very complicate due to the choice of spacetime coordinates. As before, we can always choose a suitable (spatial) coordinate system such that
the matrix (e.g. the Hessian matrix or the second fundamental form of the level sets) is diagonal at each fixed point, which can simplify
the calculations. For the spacetime Hessian matrix and the fundamental forms of the spacetime level sets, we cannot diagonalize them by
choosing the spatial coordinates, and the parabolic equation may change the form if we rotate the spacetime coordinates. In \cite{HM, CH13, CMS},
the calculations are based on a good choice of spatial coordinates and the fixed equation form, but it is very hard. In fact, the difficulty is
essential.

First, we give the definition of the spacetime convexity of a function $u(x,t)$.
\begin{definition}
Suppose $u \in C^{2} (\Omega \times (0,T] )$, where $ \Omega$ is a
domain in $\mathbb{R}^n $, then $u$ is spacetime convex if $u$ is
 convex for every $(x,t) \in \Omega \times (0,T] $, i.e
$$
D_{x,t}^2 u = \left( {\begin{matrix}
   {D^2 u} & {(Du_t )^T }  \\
   {Du_t } & {u_{tt} }  \\
 \end{matrix} } \right) \geq 0
$$
\end{definition}

In this paper and a successive paper, we will rotate the spacetime coordinates so that the spacetime matrix is diagonal at each fixed point as before,
but there are other difficulties in the calculations. We introduce some new techniques and inequalities to overcome the difficulties.
Our main result is the following constant rank theorem, which is also the main theorem of \cite{CH13}.
\begin{theorem}\label{th1.3}
Suppose $\Omega$ is a domain in $\mathbb{R}^n $, $ F=F(A,p,u,x,t)
\in C^{2,1} (\mathcal{S}_+^n \times \mathbb{R}^n \times \mathbb{R}
\times \Omega \times (0,T])$ and $u \in C^{2,1}(\Omega \times
(0,T])$ is a spacetime convex solution of \eqref{1.4}. If $F$
satisfies $(\ref{1.5})$ and the following condition
\begin{eqnarray}\label{1.7}
F(A^{-1},p,u,x,t)\quad \text{ is locally convex in } (A,u,x,t)
\text{ for each fixed } p \in \mathbb{R}^n,
\end{eqnarray}
then $D_{x,t}^2 u$ has constant rank in $\Omega$ for each
fixed $t \in (0,T] $. Moreover, let $l(t)$ be the minimal rank of
$D_{x,t}^2 u$ in $\Omega$, then $l(s) \leqslant l(t)$ for all
$0< s \leqslant t \leqslant T$.
\end{theorem}

As it is well known, one needs to choose a
suitable coordinate system in order to to simplify the calculations in
the proof of a constant rank theorem. In \cite{HM, CH13}, the proof is based on the coordinate system such that $D_x^2 u$ is diagonalized at each point. In this paper, we give another proof using a coordinate system such that $D_{x,t}^2 u$ is diagonalized at each point.

\begin{remark}
The same techniques can be used to obtain a Constant Rank Theorem of spacetime second fundamental form of space-time convex level sets.
\end{remark}

The rest of the paper is organized as follows. Section
2 contains some preliminaries. In Section 3, we prove a special case: heat equation, introduce the new ideas and new difficulties.
In Section 4, we give a new simple proof of the Constant Rank Theorem \ref{th1.3}.

\section{PRELIMINARIES}

In this section, we do some preliminaries.

\subsection{ elementary symmetric functions}

First, we recall the definition and some basic properties
of elementary symmetric functions, which could be found in
\cite{GM03,L96}.

\begin{definition}
For any $k = 1, 2,\cdots, n,$ we set
$$\sigma_k(\lambda) = \sum _{1 \le i_1 < i_2 <\cdots<i_k\leq n}\lambda_{i_1}\lambda_{i_2}\cdots\lambda_{i_k},
 \qquad \text {for any} \quad\lambda=(\lambda_1,\cdots,\lambda_n)\in {\Bbb R}^n.$$
 We also set $\sigma_0=1$ and $\sigma_k =0$ for $k>n$.
\end{definition}

We denote by $\sigma _k (\lambda \left| i \right.)$ the symmetric
function with $\lambda_i = 0$ and $\sigma _k (\lambda \left| ij
\right.)$ the symmetric function with $\lambda_i =\lambda_j = 0$.

The definition can be extended to symmetric matrices by letting
$\sigma_k(W) = \sigma_k(\lambda(W))$, where $ \lambda(W)= (\lambda
_1(W),\lambda _2 (W), \cdots ,\lambda _{n}(W))$ are the eigenvalues
of the symmetric matrix $W$. We also denote by $\sigma _k (W \left|
i \right.)$ the symmetric function with $W$ deleting the $i$-row and
$i$-column and $\sigma _k (W \left| ij \right.)$ the symmetric
function with $W$ deleting the $i,j$-rows and $i,j$-columns. Then
we have the following identities.

\begin{proposition}\label{prop2.2}
Suppose $W=(W_{ij})$ is diagonal, and $m$ is a positive integer,
then
\begin{align}\label{2.1}
\frac{{\partial \sigma _m (W)}} {{\partial W_{ij} }} = \begin{cases}
\sigma _{m - 1} (W\left| i \right.), &\text{if } i = j, \\
0, &\text{if } i \ne j.
\end{cases}
\end{align}
and
\begin{align}\label{2.2}
\frac{{\partial ^2 \sigma _m (W)}} {{\partial W_{ij} \partial W_{kl}
}} =\begin{cases}
\sigma _{m - 2} (W\left| {ik} \right.), &\text{if } i = j,k = l,i \ne k,\\
- \sigma _{m - 2} (W\left| {ik} \right.), &\text{if } i = l,j = k,i \ne j,\\
0, &\text{otherwise }.
\end{cases}
\end{align}
\end{proposition}

We need the following standard formulas of elementary symmetric
functions.
\begin{proposition}\label{prop2.3}
Let $\lambda=(\lambda_1,\dots,\lambda_n)\in\mathbb{R}^n$ and $k
=0, 1,...,n,$ then
\begin{align}
\label{e0.1}&\sigma_k(\lambda)=\sigma_k(\lambda|i)+\lambda_i\sigma_{k-1}(\lambda|i), \quad\forall \,1\leq i\leq n,&\\
\label{e0.2}&\sum_i \lambda_i\sigma_{k-1}(\lambda|i)=k\sigma_{k}(\lambda),&\\
\label{e0.3}&\sum_i\sigma_{k}(\lambda|i)=(n-k)\sigma_{k}(\lambda).&
\end{align}
\end{proposition}

\subsection{rank of spacetime Hessian}

To study the rank of the spacetime Hessian $D_{x,t}^2 u$, we
need the following simple lemma.
\begin{lemma} \label{lem2.4}
Suppose $D_{x,t}^2 u \geq 0$, and $l=rank\{ D_{x,t}^2 u(x_0,t_0)\}$, and
$D^2u(x_0,t_0)$ is diagonal with $ u_{11} \geq u_{22} \geq \cdots
\geq u_{nn} $, then at $(x_0,t_0)$,there is a positive constant $
C_0$ such that

CASE 1:
\begin{eqnarray*}
&&u_{11}  \geq \cdots \geq u_{l-1l-1} \geq
C_0 , \quad u_{ll} = \cdots = u_{nn} =0 , \\
&&u_{tt} -\sum\limits_{i =1}^{l-1} {\frac{{u_{it} ^2 }} {{u_{ii}
}}} \geq C_0 ,  \quad u_{it} = 0, \quad l \leqslant i \leqslant n.
\end{eqnarray*}
In particular, $\sigma_l(D^2u(x_0,t_0))= 0$.

CASE 2:
\begin{eqnarray*}
&&u_{11}  \geq \cdots \geq u_{ll} \geq C_0 , \quad u_{l+1l+1} =
\cdots = u_{nn} =0, \\
&& u_{tt}  = \sum\limits_{i = 1}^{l}{\frac{{u_{it} ^2 }} {{u_{ii} }}}
,  \quad u_{it} = 0, \quad  l+1 \leqslant i \leqslant n.
\end{eqnarray*}
In particular, $\sigma_l(D^2u(x_0,t_0)) > 0$.
\end{lemma}
Proof: Set $M=D^2u(x_0,t_0)=diag(u_{11}, u_{22}, \cdots, u_{nn}) \geq 0$ and we can assume $Rank \{ M \}= k$,
then we can obtain $k=l-1$ or $k =l$. Otherwise, if $k <l-1$, we know
\begin{align*}
 u_{l-1 l-1}= \cdots= u_{nn} =0  \text{ at }(x_0, t_0),
\end{align*}
and from $D_{x,t}^2 u  (x_0,t_0) \geq 0$, we get
\begin{align*}
u_{l-1 t}= \cdots=u_{nt} =0  \text{ at }(x_0, t_0).
\end{align*}
So $Rank \{ D_{x,t}^2 u  \} \leq l-1$, contradiction. If $k>l$, we have
\begin{align*}
l= Rank \{D_{x,t}^2 u \} \geq Rank \{ M \} =k \geq l+1   \text{ at }(x_0, t_0).
\end{align*}
This is impossible.

For $k =l-1$, we know at $(x_0, t_0)$
\begin{align*}
u_{11}  \geq \cdots \geq u_{l-1l-1} > 0  , \quad u_{ll} = \cdots = u_{nn} =0 ,
\end{align*}
and due to $D_{x,t}^2 u (x_0,t_0) \geq 0$, we get
\begin{align*}
u_{lt}= \cdots = u_{n t} = 0.
\end{align*}
Since $Rank \{ D_{x,t}^2 u \}=l$, then $\sigma_l (D_{x,t}^2 u) >0$. Direct computation yields
\begin{align*}
\sigma_l (D_{x,t}^2 u) = u_{tt}\sigma_{l-1}(M)
-\sum_{i=1}^{l-1}u_{ti}u_{it}\sigma_{l-2}(M|i) = \sigma_{l-1}(M) [u_{tt} - \sum_{i=1}^{l-1}\frac{u_{it}^2}{u_{ii}} ] >0,
\end{align*}
so we have
\begin{align*}
u_{tt} - \sum_{i=1}^{l-1}\frac{u_{in}^2}{u_{ii}} >0,
\end{align*}
This is CASE 1.

For $k =l$, we know at $(x_0, t_0)$
\begin{align*}
u_{11}  \geq \cdots \geq u_{ll} > 0  , \quad u_{l+1 l+1} = \cdots = u_{nn} =0 ,
\end{align*}
and due to $D_{x,t}^2 u (x_0,t_0) \geq 0$, we get
\begin{align*}
u_{l+1 t}= \cdots = u_{n t} = 0.
\end{align*}
Since $Rank \{ D_{x,t}^2 u \}=l$, then $\sigma_{l+1} ( D_{x,t}^2 u ) = 0$. Direct computation yields
\begin{align*}
\sigma_{l+1} ( D_{x,t}^2 u )= u_{tt}\sigma_{l}(M)
-\sum_{i=1}^{l}u _{ti} u_{it}\sigma_{l-1}(M|i) = \sigma_{l}(M) [u_{tt} - \sum_{i=1}^{l}\frac{u_{it}^2}{u_{ii}} ] =0,
\end{align*}
so we have
\begin{align*}
u_{tt} - \sum_{i=1}^{l}\frac{u_{it}^2}{u_{ii}} =0,
\end{align*}
This is CASE 2. \qed

\subsection{structure condition \eqref{1.7}}

Now we discuss the structure condition \eqref{1.7}.
For any given $ \widetilde X = ((X_{\alpha\beta} ),Y,(Z_i), D) \in
\mathcal{S}^n \times \mathbb{R} \times \mathbb{R}^{n}\times
\mathbb{R} $, we define a quadratic form

\begin{align} \label{2.6}
Q^* (\widetilde X,\widetilde X)=&\sum\limits_{a,b,c,d = 1}^n
{F^{ab,cd} X_{ab} X_{cd} }+  2\sum\limits_{a,b,c,d = 1}^n
{F^{ab} A^{cd} X_{ad} X_{bc}} + 2\sum\limits_{a,b = 1}^n {F^{ab,u}
X_{ab} Y} \notag \\
& + 2\sum\limits_{a,b = 1}^n {\sum\limits_{i = 1}^{n}
{F^{ab,x_i } X_{ab} Z_i } } + 2\sum\limits_{a,b = 1}^n {F^{ab,t } X_{ab} D }+ F^{u,u} Y^2  +
2\sum\limits_{i = 1}^{n} {F^{u,x_i } Y Z_i } \notag \\
& +2 F^{u,t} Y D + \sum\limits_{i,j = 1}^{n} {F^{x_i ,x_j } Z_i Z_j } + 2
\sum\limits_{i = 1}^{n} {F^{x_i ,t} Z_i D } + F^{t,t} D^2,
\end{align}
where the derivative functions of $F$ are evaluated at $(A,p,u,x,t)
$ and $(A^{ab})=A^{-1}$.

Through direct calculations, we can get

\begin{lemma} \label{lem2.5}
$F$ satisfies the condition \eqref{1.7} if and only if for each $p \in \mathbb{R}^{n}$
\begin{align} \label{2.7}
Q^* (\widetilde X,\widetilde X) \geqslant 0, \quad \forall \quad
\widetilde X = ((X_{\alpha\beta} ),Y,(Z_i), D) \in
\mathcal{S}^n \times \mathbb{R} \times \mathbb{R}^{n}\times
\mathbb{R},
\end{align}
where the derivative functions of $F$ are evaluated at $(A,p,u,x,t) $, and $Q^*$ is defined in
\eqref{2.6}.
\end{lemma}

The proof of Lemma \ref{lem2.5} is similar to the discussion in \cite{BG09}, and we omit it.

\subsection{an auxiliary lemma}

Similarly to the Lemma 2.5 in Bian-Guan\cite{BG09}, we have

\begin{lemma}\label{lem2.6}
Suppose $W(x)=(W_{ij}(x))\geq 0$ for every $x \in \Omega \subset
\mathbb{R}^{n}$, and $W_{ij} (x) \in C^{1,1} (\Omega )$, then for
every $\mathcal {O} \subset \subset \Omega$, there exists a positive
constant $C$ depending only on the $dist\{\mathcal {O},
\partial \Omega \}$ and $\left\| W \right\|_{C^{1,1} (\Omega )}$
such that
\begin{equation}\label{2.8}
\left| {\nabla W_{ij} } \right| \leqslant C(W_{ii} W_{jj} )^{\frac{1}{4}},
\end{equation}
for every $x \in \Omega$ and $1 \leq i, j \leq n$.
\end{lemma}

\textbf{Proof:} The same arguments as in the proof of Lemma 2.5 in
\cite{BG09} carry through with a small modification since $W$ is a
general matrix instead of a Hessian of a convex function.

It's known that for any nonnegative $C^{1,1}$ function $h$, $|\nabla
h(x)| \leq Ch ^{\frac{1}{2}} (x)$ for all $x \in \mathcal {O}$,
where $C$ depends only on $||h||_{C^{1,1}(\Omega)} $ and
$dist\{\mathcal {O}, \partial \Omega\}$ (see \cite{Tr71}).

Since $W (x) \geq 0$, so we choose $h(x) =W_{ii}(x) \geq 0$. Then we can get from the above argument
\begin{equation*}
\left| {\nabla W_{ii} } \right| \leqslant C_1 (W_{ii})^{\frac{1}{2}} = C_1 (W_{ii} W_{ii} )^{\frac{1}{4}},
\end{equation*}
\eqref{2.8} holds for $i=j$.

Similarly, for $i\ne j$, we choose $ h=\sqrt{W_{ii}W_{jj}} \geq 0 $, then we get
 \begin{equation}\label{2.9}
\left| {\nabla \sqrt{W_{ii}W_{jj}} } \right| \leqslant C_2 (\sqrt{W_{ii}W_{jj}} )^{\frac{1}{2}} = C_2 (W_{ii} W_{jj} )^{\frac{1}{4}}.
\end{equation}
And for $h=\sqrt{W_{ii}W_{jj}} - W_{ij}$,  we have
 \begin{equation}\label{2.10}
\left| {\nabla (\sqrt{W_{ii}W_{jj}} - W_{ij} )} \right| \leqslant C_3 (\sqrt{W_{ii}W_{jj}} - W_{ij} )^{\frac{1}{2}} \leq C_3 (W_{ii} W_{jj} )^{\frac{1}{4}}.
\end{equation}
So from \eqref{2.9} and \eqref{2.10}, we get
\begin{align*}
\left| {\nabla W_{ij}} \right|  =& \left| {\nabla \sqrt{W_{ii}W_{jj}} } - {\nabla (\sqrt{W_{ii}W_{jj}} - W_{ij} )} \right|  \\
\leq& \left| {\nabla \sqrt{W_{ii}W_{jj}} } \right| + \left| {\nabla (\sqrt{W_{ii}W_{jj}} - W_{ij} )} \right|  \\
\leq& (C_2 +C_3) (W_{ii} W_{jj} )^{\frac{1}{4}}.
\end{align*}
So \eqref{2.8} holds for $i \ne j$. \qed

\section{The constant rank theorem for the heat equation}

In this section, we consider a special case: heat equation. This is a new proof of the main result in \cite{HM}, and the idea is from \cite{GZ}.

Our main result is the following theorem.
\begin{theorem}\label{th3.1}
Suppose $\Omega$ is a domain in $\mathbb{R}^n $, and $u \in C^{2,1}(\Omega \times
(0,T])$ is a spacetime convex solution of the heat equation
\begin{align}\label{3.1}
 u_t  = \Delta u \quad \text{ in } \Omega \times (0,T].
\end{align}
Then $D_{x,t}^2 u$ has a constant rank in $\Omega$ for each
fixed $t \in (0,T] $. Moreover, let $l(t)$ be the minimal rank of
$D_{x,t}^2 u$ in $\Omega$, then $l(s) \leqslant l(t)$ for all
$0< s \leqslant t \leqslant T$.
\end{theorem}

\begin{proof}
Following the assumptions of Theorem \ref{th3.1}, we know $D_{x,t}^2 u \geq 0$. By the regularity theory, we can get
$u \in C^{3,1}(\Omega \times (0,T])$. Suppose $D_{x,t}^2 u$ attains its minimal rank $l$ at some point $(x_0, t_0) \in \Omega \times (0, T]$.
We pick a small open neighborhood $\mathcal {O}\times (t_0-\delta,
t_0]$ of $(x_0, t_0)$. And for any fixed point $(x, t) \in \mathcal {O}\times
(t_0-\delta, t_0]$, we rotate coordinates $(x,t)$ with
\begin{align}
 y = (y_1 , \cdots ,y_n ,y_{n + 1} ) = (x,t) P , \notag
 \end{align}
such that the matrix $D_{x,t}^2 u$ is diagonal, where
$P = \left( {P_{\alpha \beta } } \right)_{n + 1 \times n + 1}$ is an orthogonal matrix.
For convenience, we will use $i,j,k,l =1, \cdots, n$ to represent the $x$ coordinates, $t$ still the time coordinate, and
 $\alpha, \beta, \gamma, \eta =1, \cdots, n+1$ the $y$ coordinates. And we have
\begin{align}
\label{3.2}&\frac{{\partial y_\alpha  }}{{\partial x_i }} = P_{i\alpha },  \\
\label{3.3}&\frac{{\partial y_\alpha  }}{{\partial t}} = P_{n + 1\alpha },
\end{align}
In the following, we always denote
\begin{align*}
& u_{i}=\frac{{\partial u}}{{\partial x_i }}, u_{t}=\frac{{\partial u}}{{\partial t }} , u_{\alpha}=\frac{{ \partial u}}{{\partial y_\alpha}}, u_{n+1}=\frac{{ \partial u}}{{\partial y_{n+1}}}, \\
& u_{ij}=\frac{{\partial^2 u}}{{\partial x_i \partial x_j}}, u_{it}=\frac{{\partial^2 u}}{{\partial x_i \partial t }}, u_{tt}=\frac{{\partial^2 u}}{{ \partial t ^2 }}, u_{i\alpha}=\frac{{\partial^2 u}}{{\partial x_i \partial y_\alpha}}, \\
&u_{\alpha t}=\frac{{\partial^2 u}}{{\partial y_\alpha \partial t}},u_{\alpha \beta}=\frac{{ \partial^2 u}}{{\partial y_\alpha \partial y_\beta}}, \text{ etc. }
\end{align*}

At $(x,t)$, the matrix $D_{x,t}^2 u$ is diagonal in the $y$ coordinates, so without loss of generality we assume $ u_{\alpha \alpha} \geq u_{\beta \beta}$ for arbitrary $1 \leq \alpha < \beta \leq n+1$. Then there is a positive
constant $C > 0$ depending only on $\left\| u \right\|_{C^{3,1} }$, such that $ \frac{{ \partial^2 u}}{{\partial y_1 \partial y_1}} \geqslant \cdots \geqslant
\frac{{ \partial^2 u}}{{\partial y_l \partial y_l}} \geqslant C > 0 $ for all $(x, t) \in \mathcal {O}\times (t_0-\delta, t_0]$.
For convenience we denote $ G = \{1, \cdots ,l\} $ and $ B =
\{l+1, \cdots ,n, n+1 \} $ which means good terms and bad ones in
indices respectively. Without confusion we will also simply denote $ G = \{ \frac{{ \partial^2 u}}{{\partial y_1 \partial y_1}}, \cdots ,\frac{{ \partial^2 u}}{{\partial y_l \partial y_l}} \} $ and $B = \{ \frac{{ \partial^2 u}}{{\partial y_{l+1} \partial y_{l+1}}} , \cdots ,\frac{{ \partial^2 u}}{{\partial y_{n} \partial y_{n}}},\frac{{ \partial^2 u}}{{\partial y_{n+1} \partial y_{n+1}}} \} $.

Set
\begin{align}\label{3.4}
\phi  = \sigma _{l + 1} (D_{x,t}^2 u) ,
\end{align}
In the following, we will prove a differential inequality
\begin{align}\label{3.5}
\Delta _x \phi  - \phi _t \le  C(\phi  + |\nabla _x \phi |) \quad \text{ in } \mathcal {O}\times (t_0-\delta,
t_0].
\end{align}
Then by the strong maximum principle and the method of continuity, we can prove the theorem.

In the $y$ coordinates, we have
\begin{align}
 \phi  = \sigma _{l + 1} (D_y^2 u) \ge \sigma _l (G)\sigma _1 (B) \ge 0,  \notag
 \end{align}
so we get
\begin{align}\label{3.6}
 u_{\alpha \alpha }  = O(\phi ),\alpha  \in B .
 \end{align}
Taking the first derivatives of $\phi$ in $x$, $t$, we have
\begin{align}
\label{3.7}&\phi _i  = \frac{{\partial \phi }}{{\partial x_i }} = \sum\limits_{\alpha  = 1}^{n + 1} {\sigma _l (D_y^2 u|\alpha )u_{\alpha \alpha i} }  = \sigma _l (G)\sum\limits_{\alpha  \in B} {u_{\alpha \alpha i} }  + O(\phi ), \\
\label{3.8}&\phi _t  = \frac{{\partial \phi }}{{\partial t}} = \sum\limits_{\alpha  = 1}^{n + 1} {\sigma _l (D_y^2 u|\alpha )u_{\alpha \alpha t} }  = \sigma _l (G)\sum\limits_{\alpha  \in B} {u_{\alpha \alpha t} }  + O(\phi ),
\end{align}
so from \eqref{3.7}, we get
\begin{align}\label{3.9}
\sum\limits_{\alpha  \in B} {u_{\alpha \alpha i} } = O(\phi + |\nabla_x \phi| ).
\end{align}
Taking the second derivatives of $\phi$ in $y$ coordinates, we have
\begin{align}\label{3.10}
\phi _{\alpha \beta }  =& \frac{{\partial ^2 \phi }}{{\partial y_\alpha  \partial y_\beta  }} \notag \\
=& \sum\limits_{\gamma  = 1}^{n + 1} {\frac{{\partial \sigma _{l + 1} (D_y^2 u)}}{{\partial u_{\gamma \gamma } }}u_{\gamma \gamma \alpha \beta } }  + \sum\limits_{\gamma  \ne \eta } {\frac{{\partial ^2 \sigma _{l + 1} }}{{\partial u_{\gamma \gamma } \partial u_{\eta \eta } }}u_{\gamma \gamma \alpha } u_{\eta \eta \beta } }  + \sum\limits_{\gamma  \ne \eta } {\frac{{\partial ^2 \sigma _{l + 1} }}{{\partial u_{\gamma \eta } \partial u_{\eta \gamma } }}u_{\gamma \eta \alpha } u_{\eta \gamma \beta } }  \notag \\
=& \sum\limits_{\gamma  = 1}^{n + 1} {\sigma _l (D_y^2 u|\gamma )u_{\gamma \gamma \alpha \beta } }  + \sum\limits_{\gamma  \ne \eta } {\sigma _{l - 1} (D_y^2 u|\gamma \eta )u_{\gamma \gamma \alpha } u_{\eta \eta \beta } }  - \sum\limits_{\gamma  \ne \eta } {\sigma _{l - 1} (D_y^2 u|\gamma \eta )u_{\gamma \eta \alpha } u_{\eta \gamma \beta } }  ,
\end{align}
where
\begin{align}\label{3.11}
\sum\limits_{\gamma  = 1}^{n + 1} {\sigma _l (D_y^2 u|\gamma )u_{\gamma \gamma \alpha \beta } }  =& \sum\limits_{\gamma  \in B} {\sigma _l (D_y^2 u|\gamma )u_{\gamma \gamma \alpha \beta } }  + \sum\limits_{\gamma  \in G} {\sigma _l (D_y^2 u|\gamma )u_{\gamma \gamma \alpha \beta } }  \notag \\
=& \sigma _l (G)\sum\limits_{\gamma  \in B} {u_{\gamma \gamma \alpha \beta } }  + O(\phi ),
\end{align}
\begin{align}\label{3.12}
\sum\limits_{\gamma  \ne \eta } {\sigma _{l - 1} (D_y^2 u|\gamma \eta )u_{\gamma \gamma \alpha } u_{\eta \eta \beta } }  =& \sum\limits_{\scriptstyle \gamma \eta  \in B \hfill \atop  \scriptstyle \gamma  \ne \eta  \hfill} {\sigma _{l - 1} (D_y^2 u|\gamma \eta )u_{\gamma \gamma \alpha } u_{\eta \eta \beta } }  + \sum\limits_{\scriptstyle \gamma  \in B \hfill \atop  \scriptstyle \eta  \in G \hfill} {\sigma _{l - 1} (D_y^2 u|\gamma \eta )u_{\gamma \gamma \alpha } u_{\eta \eta \beta } }  \notag \\
&+ \sum\limits_{\scriptstyle \gamma  \in G \hfill \atop  \scriptstyle \eta  \in B \hfill} {\sigma _{l - 1} (D_y^2 u|\gamma \eta )u_{\gamma \gamma \alpha } u_{\eta \eta \beta } }  + \sum\limits_{\scriptstyle \gamma \eta  \in G \hfill \atop  \scriptstyle \gamma  \ne \eta  \hfill} {\sigma _{l - 1} (D_y^2 u|\gamma \eta )u_{\gamma \gamma \alpha } u_{\eta \eta \beta } }   \notag\\
=& O(\phi ) + \sum\limits_{\eta  \in G} {\sigma _{l - 1} (G|\eta )u_{\eta \eta \beta } } \sum\limits_{\gamma  \in B} {u_{\gamma \gamma \alpha } }  + \sum\limits_{\gamma  \in G} {\sigma _{l - 1} (G|\gamma )u_{\gamma \gamma \alpha } } \sum\limits_{\eta  \in B} {u_{\eta \eta \beta } }   \notag\\
=& \sigma _l (G)[\sum\limits_{\eta  \in G} {\frac{{u_{\eta \eta \beta } }}{{u_{\eta \eta } }}} \sum\limits_{\gamma  \in B} {u_{\gamma \gamma \alpha } }  + \sum\limits_{\gamma  \in G} {\frac{{u_{\gamma \gamma \alpha } }}{{u_{\gamma \gamma } }}} \sum\limits_{\eta  \in B} {u_{\eta \eta \beta } } ] + O(\phi ),
\end{align}
and
\begin{align}\label{3.13}
\sum\limits_{\gamma  \ne \eta } {\sigma _{l - 1} (D_y^2 u|\gamma \eta )u_{\gamma \eta \alpha } u_{\eta \gamma \beta } }  =& \sum\limits_{\scriptstyle \gamma \eta  \in B \hfill \atop \scriptstyle \gamma  \ne \eta  \hfill} {\sigma _{l - 1} (D_y^2 u|\gamma \eta )u_{\gamma \eta \alpha } u_{\eta \gamma \beta } }  + \sum\limits_{\scriptstyle \gamma  \in B \hfill \atop  \scriptstyle \eta  \in G \hfill} {\sigma _{l - 1} (D_y^2 u|\gamma \eta )u_{\gamma \eta \alpha } u_{\eta \gamma \beta } }  \notag \\
&+ \sum\limits_{\scriptstyle \gamma  \in G \hfill \atop \scriptstyle \eta  \in B \hfill} {\sigma _{l - 1} (D_y^2 u|\gamma \eta )u_{\gamma \eta \alpha } u_{\eta \gamma \beta } }  + \sum\limits_{\scriptstyle \gamma \eta  \in G \hfill \atop \scriptstyle \gamma  \ne \eta  \hfill} {\sigma _{l - 1} (D_y^2 u|\gamma \eta )u_{\gamma \eta \alpha } u_{\eta \gamma \beta } }  \notag \\
=& O(\phi ) + \sum\limits_{\scriptstyle \gamma  \in B \hfill \atop \scriptstyle \eta  \in G \hfill} {\sigma _{l - 1} (G|\eta )u_{\gamma \eta \alpha } u_{\eta \gamma \beta } }  + \sum\limits_{\scriptstyle \gamma  \in G \hfill \atop \scriptstyle \eta  \in B \hfill} {\sigma _{l - 1} (G|\gamma )u_{\gamma \eta \alpha } u_{\eta \gamma \beta } }  \notag \\
=& 2\sigma _l (G)\sum\limits_{\scriptstyle \gamma  \in B \hfill \atop  \scriptstyle \eta  \in G \hfill} {\frac{{u_{\gamma \eta \alpha } u_{\eta \gamma \beta } }}{{u_{\eta \eta } }}}  + O(\phi ).
\end{align}
So from \eqref{3.10}-\eqref{3.13}, we get
\begin{align}\label{3.14}
\phi _{\alpha \beta }  =& \sigma _l (G)\sum\limits_{\gamma  \in B} {u_{\gamma \gamma \alpha \beta } }  - 2\sigma _l (G)\sum\limits_{\scriptstyle \gamma  \in B \hfill \atop \scriptstyle \eta  \in G \hfill} {\frac{{u_{\gamma \eta \alpha } u_{\eta \gamma \beta } }}{{u_{\eta \eta } }}} \notag \\
&+ \sigma _l (G)[\sum\limits_{\eta  \in G} {\frac{{u_{\eta \eta \beta } }}{{u_{\eta \eta } }}} \sum\limits_{\gamma  \in B} {u_{\gamma \gamma \alpha } }  + \sum\limits_{\gamma  \in G} {\frac{{u_{\gamma \gamma \alpha } }}{{u_{\gamma \gamma } }}} \sum\limits_{\eta  \in B} {u_{\eta \eta \beta } } ] + O(\phi ).
\end{align}
Then we get
\begin{align}
\Delta _x \phi  =& \sum\limits_{i = 1}^n { \phi _{ii} } =\sum\limits_{i = 1}^n {\sum\limits_{\alpha \beta  = 1}^{n + 1} {P_{i\alpha } P_{i\beta } \phi _{\alpha \beta } } }   \notag\\
=& \sigma _l (G)\sum\limits_{i = 1}^n {\sum\limits_{\gamma  \in B} {[\sum\limits_{\alpha \beta  = 1}^{n + 1} {P_{i\alpha } P_{i\beta } } u_{\gamma \gamma \alpha \beta } ]} }  - 2\sigma _l (G)\sum\limits_{i = 1}^n {\sum\limits_{\scriptstyle \gamma  \in B \hfill \atop \scriptstyle \eta  \in G \hfill} {\frac{{[\sum\limits_{\alpha  = 1}^{n + 1} {P_{i\alpha } u_{\gamma \eta \alpha } } ][\sum\limits_{\beta  = 1}^{n + 1} {P_{i\beta } u_{\eta \gamma \beta } } ]}}{{u_{\eta \eta } }}} }  \notag \\
&+ \sigma _l (G)\sum\limits_{i = 1}^n {[\sum\limits_{\eta  \in G} {\frac{{\sum\limits_{\beta  = 1}^{n + 1} {P_{i\beta } } u_{\eta \eta \beta } }}{{u_{\eta \eta } }}} \sum\limits_{\gamma  \in B} {\sum\limits_{\alpha  = 1}^{n + 1} {P_{i\alpha } } u_{\gamma \gamma \alpha } }  + \sum\limits_{\gamma  \in G} {\frac{{\sum\limits_{\alpha  = 1}^{n + 1} {P_{i\alpha } } u_{\gamma \gamma \alpha } }}{{u_{\gamma \gamma } }}} \sum\limits_{\eta  \in B} {\sum\limits_{\beta  = 1}^{n + 1} {P_{i\beta } } u_{\eta \eta \beta } } ]}  + O(\phi )  \notag\\
=& \sigma _l (G)\sum\limits_{i = 1}^n {\sum\limits_{\gamma  \in B} {u_{\gamma \gamma ii} } }  - 2\sigma _l (G)\sum\limits_{i = 1}^n {\sum\limits_{\scriptstyle \gamma  \in B \hfill \atop \scriptstyle \eta  \in G \hfill} {\frac{{u_{\gamma \eta i} u_{\eta \gamma i} }}{{u_{\eta \eta } }}} }  + 2\sigma _l (G)\sum\limits_{i = 1}^n {\sum\limits_{\eta  \in G} {\frac{{u_{\eta \eta i} }}{{u_{\eta \eta } }}} \sum\limits_{\gamma  \in B} {u_{\gamma \gamma i} } }  + O(\phi ). \notag
\end{align}
By \eqref{3.9}, it holds
\begin{align}\label{3.15}
\Delta _x \phi  = \sigma _l (G)\sum\limits_{\gamma  \in B} {\left[ {\Delta _x u_{\gamma \gamma }  - 2\sum\limits_{\eta  \in G} {\sum\limits_{i = 1}^n {\frac{{u_{\gamma \eta i} ^2 }}{{u_{\eta \eta } }}} } } \right]}  + O(\phi  + |\nabla _x \phi |)
\end{align}
By \eqref{3.8}, \eqref{3.15} and the equation \eqref{3.1}, we can obtain
\begin{align}\label{3.16}
\Delta _x \phi  - \phi _t  =& \sigma _l (G)\sum\limits_{\gamma  \in B} {\left[ {\left( {\Delta _x u_{\gamma \gamma }  - u_{\gamma \gamma t} } \right) - 2\sum\limits_{\eta  \in G} {\sum\limits_{i = 1}^n {\frac{{u_{\gamma \eta i} ^2 }}{{u_{\eta \eta } }}} } } \right]}  + O(\phi  + |\nabla _x \phi |) \notag \\
=&  - 2\sigma _l (G)\sum\limits_{\gamma  \in B} {\sum\limits_{\eta  \in G} {\sum\limits_{i = 1}^n {\frac{{u_{\gamma \eta i} ^2 }}{{u_{\eta \eta } }}} } }  + O(\phi  + |\nabla _x \phi |)  \notag\\
\le& C(\phi  + |\nabla _x \phi |).
\end{align}
Then \eqref{3.5} holds, and we prove Theorem \ref{th3.1}.
\end{proof}

\begin{remark}
In the proof, we rotate the spacetime coordinates $(x, t)$ such that $D_{x,t}^2 u$ is diagonal, and the heat equation in $(x, t)$
is a new linear equation in $y$ coordinates. While,  the idea in \cite{HM} is only to rotate the spatial coordinates $x$ such that $D_{x}^2 u$ is diagonal, and keep the heat equation invariant. This proof is easier than \cite{HM}, and it can be applied to linear
parabolic equations.
\end{remark}
\begin{remark}
For fully nonlinear parabolic equations, the above test function $\phi =\sigma _{l + 1} (D_{x,t}^2 u)$ is not good enough. If we choose Bian-Guan's test
function $\phi =\sigma _{l + 1} (D_{x,t}^2 u)+ \frac{\sigma _{l +2} (D_{x,t}^2 u)}{\sigma _{l + 1} (D_{x,t}^2 u)}$ as in \cite{BG09} and rotate the spacetime coordinates $(x, t)$ such that $D_{x,t}^2 u$ is diagonal, we should use
$|\nabla_x \phi|$, not $|\nabla_y \phi |=|D_{x,t} u|$ to control the bad terms $\sum\limits_{\alpha \beta \in B} {|u_{\alpha \beta \gamma|} }$. This is very easy.
\end{remark}

\section{Proof of Theorem~\ref{th1.3}}

In this section, we will prove the constant rank theorem of spacetime Hessian for the fully nonlinear parabolic equations, Theorem \ref{th1.3}. In fact, we will use a new idea to simplify the calculations. The key is to use
the constant rank properties of the spatial Hessian $ D^2 u$.

Suppose $W(x,t)=D_{x,t}^2 u$ attains the minimal
rank $l$ at some point $(x_0, t_0) \in \Omega \times (0, T]$. We may assume
$l\leqslant n$, otherwise there is nothing to prove. From lemma
\ref{lem2.4}, we can transform the $x$ coordinates such that $D^2u(x_0,t_0)$ is diagonal with $ u_{11} \geq u_{22} \geq \cdots
\geq u_{nn} $, then at $(x_0,t_0)$, there is a positive constant $
C_0$ such that

CASE 1:
\begin{eqnarray*}
&&u_{11}  \geq \cdots \geq u_{l-1l-1} \geq
C_0 , \quad u_{ll} = \cdots = u_{nn} =0 , \\
&&u_{tt} -\sum\limits_{i =1}^{l-1} {\frac{{u_{it} ^2 }} {{u_{ii}
}}} \geq C_0 ,  \quad u_{it} = 0, \quad l \leqslant i \leqslant n.
\end{eqnarray*}
In particular, $\sigma_l(D^2u(x_0,t_0))= 0$.

CASE 2:
\begin{eqnarray*}
&&u_{11}  \geq \cdots \geq u_{ll} \geq C_0 , \quad u_{l+1l+1} =
\cdots = u_{nn} =0, \\
&& u_{tt}  = \sum\limits_{i = 1}^{l}{\frac{{u_{it} ^2 }} {{u_{ii} }}}
,  \quad u_{it} = 0, \quad  l+1 \leqslant i \leqslant n.
\end{eqnarray*}
In particular, $\sigma_l(D^2u(x_0,t_0)) > 0$.

In the following, we denote
\begin{align*}
& F^{ij}  = \frac{{\partial F}} {{\partial u_{ij} }} ,  F^{u_i } =
\frac{{\partial F}} {{\partial u_i }} ,  F^u  = \frac{{\partial F}}
{{\partial u}} ,F^t  = \frac{{\partial F}}{{\partial t}} ,\\
&F^{ij,kl}  = \frac{{\partial ^2 F}} {{\partial u_{ij} \partial
u_{kl} }},  F^{ij, u_k}  = \frac{{\partial ^2 F}} {{\partial
u_{ij}\partial u_k }} ,  F^{ij,u}  = \frac{{\partial ^2 F}}
{{\partial u_{ij}\partial u}},\\
& F^{u_i u_j }  = \frac{{\partial ^2 F}} {{\partial u_i \partial
u_j}} ,  F^{u_i, u}  = \frac{{\partial ^2 F}} {{\partial u_i
\partial u}} ,  F^{u,u}  = \frac{{\partial ^2 F}} {{\partial
u\partial u}},
\end{align*}
where $ 1 \leqslant i, j, k, l \leqslant n $.

In order to prove Theorem \ref{th1.3}, we will firstly consider the constant rank theorem of $D^2 u$,
which is all from \cite{BG09}, and state some important results. Then we prove Theorem \ref{th1.3} under the above
CASE 1 and CASE 2, respectively.

\subsection{The constant rank properties of $D^2u$}

Following the assumptions of Theorem \ref{th1.3}, we know $D^2u \geq 0$. If $F$
satisfies \eqref{1.7}, then $F$ satisfies \eqref{1.6} and Theorem \ref{th1.1} holds.
Suppose $D^2u$ attains its minimal rank $l$ at some point $(x_0, t_0) \in \Omega \times (0, T]$. We pick
a small open neighborhood $\mathcal {O}\times (t_0-\delta,
t_0]$ of $(x_0, t_0)$, and for any fixed point $(x, t) \in \mathcal {O}\times
(t_0-\delta, t_0]$, we rotate the $x$ coordinates such that the matrix $D^2 u (x,t)$ is diagonal
and without loss of generality we assume $ u_{11} \geqslant u_{22}
\geqslant \cdots \geqslant u_{nn} $. Then there is a positive
constant $C > 0$ depending only on $\left\| u \right\|_{C^{3,1} }$, such that $ u_{11} \geqslant \cdots \geqslant
u_{ll} \geqslant C > 0 $ for all $(x, t) \in \mathcal {O}\times (t_0-\delta, t_0]$.
For convenience we denote $ G = \{1, \cdots ,l\} $ and $ B =
\{ l+1, \cdots ,n\} $ which means good terms and bad ones in
indices respectively. Without confusion we will also simply denote $ G = \{ u_{11} , \cdots ,u_{ll} \} $ and $
B = \{ u_{l+1l+1} , \cdots ,u_{nn} \} $.

Set
\begin{equation}\label{4.1}
q(W) = \left\{ \begin{matrix}
  \frac{{\sigma _{l + 2} (W)}}{{\sigma _{l + 1} (W)}}, \quad if \quad\sigma _{l + 1} (W) > 0, \hfill \cr
  0, \qquad if \quad \sigma _{l + 1} (W) = 0. \hfill \cr
 \end{matrix}  \right.
\end{equation}
And denote
\begin{equation}\label{4.2}
\phi = \sigma_{l+1}(D^2u)+q(D^2u).
\end{equation}
In \cite{BG09}, Bian-Guan got the following differential inequality,
\begin{equation} \label{4.3}
\sum\limits_{ij = 1}^n {F^{ij} \phi _{ij}(x,t)-\phi_t(x,t) }
\leqslant C(\phi (x,t) + \left| {\nabla \phi(x,t) } \right|)
-C_2\sum\limits_{i,j \in B} {\left| {\nabla u_{ij} }
\right|},
\end{equation}
where $C_1$,$C_2$ are two positive constants and
$(x,t) \in \mathcal {O} \times (t_0-\delta, t_0]$. Together with
\begin{equation}\label{4.4}
\phi(x,t) \geq 0, \quad  (x,t) \in \mathcal {O} \times (t_0-\delta, t_0], \quad
\phi(x_0, t_0) =0,
\end{equation}
we can apply the strong maximum principle of parabolic equations, and we obtain
\begin{equation}\label{4.5}
\phi(x,t) \equiv 0, \quad  (x,t) \in \mathcal {O} \times (t_0-\delta, t_0],
\end{equation}
and
\begin{equation}\label{4.6}
\sum\limits_{i,j \in B} {\left| {\nabla u_{ij} }\right|}\equiv 0, \quad  (x,t) \in \mathcal {O} \times (t_0-\delta, t_0].
\end{equation}
By the argument in \cite{BG09}, the null space of $D^2u$ is parallel for all
$x \in \Omega$, $t \in (t_0-\delta, t_0]$. So we can fix $e_{l+1}, \cdots,
e_{n}$ in $(x, t) \in \Omega \times (t_0-\delta, t_0]$such that
$u_{ii}(x,t) \equiv 0$, for any $(x,t) \in \mathcal {O} \times (t_0-\delta, t_0]$ and $i  \in B$.

So we can get the following constant rank properties.
\begin{proposition} \label{prop4.1}
Under above assumptions, we can get
\begin{eqnarray}
\label{4.7}&& u_{ii}(x,t) \equiv 0, \quad \text{ for } (x,t) \in \mathcal {O} \times (t_0-\delta, t_0] \text{ and i } \in B, \\
\label{4.8}&&\sum\limits_{i,j \in B} {\left| {\nabla u_{ij} }\right|} (x,t) \equiv 0, \quad \text{ for } (x,t) \in \mathcal {O} \times (t_0-\delta, t_0].
\end{eqnarray}
\end{proposition}

\subsection{CASE 1}

In this subsection, we will prove Theorem \ref{th1.3} under CASE 1.
Suppose $W(x,t)=D_{x,t}^2 u$ attains the minimal
rank $l$ at some point $(x_0, t_0) \in \Omega \times (0, T]$. We may assume
$l\leqslant n$, otherwise there is nothing to prove. Then from lemma
\ref{lem2.4}, there is a neighborhood $\mathcal {O}\times
(t_0-\delta, t_0]$ of $(x_0, t_0)$, such that  $ u_{11} \geq
\cdots \geq u_{l-1l-1} \geqslant C > 0 $ and $u_{tt} -\sum\limits_{i
= 1}^{l-1} {\frac{{u_{it} ^2 }} {{u_{ii} }}} \geq C$ for all $(x,t)
\in \mathcal {O}\times (t_0-\delta, t_0]$. And for any fixed
point  $(x,t) \in \mathcal {O}\times (t_0-\delta, t_0]$, we
rotate the $x$ coordinate such that the matrix $D^2
u$ is diagonal, and without loss of generality we assume $ u_{11}
\geq u_{22}\geq \cdots \geq u_{nn} $. We can denote $ G = \{ 1, \cdots , l-1 \} $ and $ B = \{ l,
\cdots, n\}$ .

In order to prove the main theorem, we just need to prove
\begin{equation} \label{4.9}
\sigma_{l+1}(D_{x,t}^2 u)\equiv 0, \quad \text{for every}\quad
(x,t) \in \mathcal {O} \times (t_0-\delta, t_0].
\end{equation}
In fact, when $D^2 u$ is diagonal at $(x,t)$, we have
\begin{align} \label{4.10}
\sigma_{l+1}(D_{x,t}^2 u)=&\sigma_{l+1}(D^2 u)+ u_{tt} \sigma_{l}(D^2 u) - \sum_{i=1}^n u_{it}^2 \sigma_{l-1}(D^2 u|i) \notag \\
\leq& \sigma_{l+1}(D^2 u)+ u_{tt} \sigma_{l}(D^2 u).
\end{align}

Under CASE 1, the spatial Hessian $D^2u$ attains the rank $l-1$.
From \cite{BG09}, the constant rank theorem holds for the spatial
Hessian $D^2u$ of the solution $u$ for the equation $
u_t=F(D^2u,Du,u,x,t)$, so we can get,
\begin{equation}\label{4.11}
\sigma_{l+1}(D^2u)=\sigma_{l}(D^2u)\equiv 0, \quad \text{for
every}\quad (x,t) \in \mathcal {O} \times (t_0-\delta, t_0].
\end{equation}
Then
\begin{align} \label{4.12}
0 \leq \sigma_{l+1}(D_{x,t}^2 u)\leq \sigma_{l+1}(D^2 u)+ u_{tt} \sigma_{l}(D^2 u) =0.
\end{align}
Hence \eqref{4.9} holds.

By the continuity method, Theorem \ref{th1.3} holds under CASE 1.

\subsection{CASE 2}

In this subsection, we will prove Theorem \ref{th1.3} under CASE 2.
Suppose $W(x,t)=D_{x,t}^2 u$ attains the minimal
rank $l$ at some point $(x_0, t_0) \in \Omega \times (0, T]$. We may assume
$l\leqslant n$, otherwise there is nothing to prove. Under CASE 2, $l$ is also the minimal rank of $D^2u$ in $\Omega \times (t_0-\delta, t_0]$,
 we obtain from the discussions in Subsection 4.1, we can fix $e_{l+1}, \cdots,
e_{n}$ in $(x, t) \in \mathcal {O} \times (t_0-\delta, t_0]$ such that
$u_{ii}(x,t) \equiv 0$, for any $(x,t) \in \mathcal {O} \times (t_0-\delta, t_0]$ and $i =l+1, \cdots, n$.
For each fixed $(x,t)\in\mathcal{O}\times(t_0-\delta, t_0]$,
we choose the coordinates $\bar e_{1},\cdots,\bar e_l, \bar e_{n+1}$ so that $D_{x,t}^2 u$
is diagonal in the coordinates $\{\bar e_{1},\cdots,\bar e_l, e_{l+1}, \cdots,
e_{n}, \bar e_{n+1}\}$ . In fact, the new coordinate is
\begin{align}\label{4.13}
 y = (y_1 , \cdots ,y_n ,y_{n + 1} ) = (x,t)P,
\end{align}
where $P$ is an orthornormal matrix with
\begin{align}\label{4.14}
P = \left( {P_{\alpha \beta } } \right)_{n + 1 \times n + 1}  = \left( {\begin{array}{*{20}c}
   {P_{11} } &  \cdots  & {P_{1l} } & 0 &  \cdots  & 0 & {P_{1n + 1} }  \\
    \vdots  &  \ddots  &  \vdots  &  \vdots  &  \ddots  &  \vdots  &  \vdots   \\
   {P_{l1} } &  \cdots  & {P_{ll} } & 0 &  \cdots  & 0 & {P_{l n  + 1} }  \\
   0 &  \cdots  & 0 & 1 &  \cdots  & 0 & 0  \\
    \vdots  &  \ddots  &  \vdots  &  \vdots  &  \ddots  &  \vdots  &  \vdots   \\
   0 &  \cdots  & 0 & 0 &  \cdots  & 1 & 0  \\
   {P_{n + 11} } &  \cdots  & {P_{n + 1l} } & 0 &  \cdots  & 0 & {P_{n + 1n + 1} }  \\
\end{array}} \right).
\end{align}
Without loss of generality, we can assume $ \frac{\partial ^2 u} {\partial y_1 \partial y_1} \geqslant \cdots \geqslant
\frac{\partial ^2 u} {\partial y_l \partial y_l} \geqslant C > 0 $ for all $(x, t) \in \mathcal {O}\times (t_0-\delta, t_0]$, where the positive
constant $C > 0$ depending only on $\left\| u \right\|_{C^{3,1} }$.
For convenience we denote $ G = \{1, \cdots , l\} $ and $ B =
\{ l+1, \cdots ,n\} $ which means good terms and bad ones in
indices respectively. Without confusion we will also simply denote $ G = \{ \frac{\partial ^2 u} {\partial y_1 \partial y_1}, \cdots ,\frac{\partial ^2 u} {\partial y_l \partial y_l} \} $ and $
B = \{ \frac{\partial ^2 u} {\partial y_{l+1} \partial y_{l+1}} , \cdots ,\frac{\partial ^2 u} {\partial y_n \partial y_n}\} $.

For simplify, we will use $i,j,k,l =1, \cdots, n$ to represent the $x$ coordinates, $t$ still the time coordinate, and
 $\alpha, \beta, \gamma, \eta =1, \cdots, n+1$ the $y$ coordinates. And we have
\begin{align}
\label{4.15}&\frac{{\partial y_\alpha  }}{{\partial x_i }} = P_{i\alpha }  \\
\label{4.16}&\frac{{\partial y_\alpha  }}{{\partial t}} = P_{n + 1\alpha }
\end{align}
In the following, we always denote
\begin{align*}
& u_{i}=\frac{{\partial u}}{{\partial x_i }}, u_{t}=\frac{{\partial u}}{{\partial t }} , u_{\alpha}=\frac{{ \partial u}}{{\partial y_\alpha}},  u_{n+1}=\frac{{ \partial u}}{{\partial y_{n+1}}},\\
& u_{ij}=\frac{{\partial^2 u}}{{\partial x_i \partial x_j}}, u_{it}=\frac{{\partial^2 u}}{{\partial x_i \partial t }}, u_{tt}=\frac{{\partial^2 u}}{{ \partial t ^2 }}, u_{i\alpha}=\frac{{\partial^2 u}}{{\partial x_i \partial y_\alpha}}, \\
&u_{\alpha t}=\frac{{\partial^2 u}}{{\partial y_\alpha \partial t}},u_{\alpha \beta}=\frac{{ \partial^2 u}}{{\partial y_\alpha \partial y_\beta}}, \text{ etc. }
\end{align*}

From the discussion in Subsection 4.1,
\begin{align}\label{4.17}
 u_{\alpha \alpha }= \frac{\partial ^2 u} {\partial y_{\alpha} \partial y_{\alpha}} = \frac{\partial ^2 u} {\partial x_{\alpha} \partial x_{\alpha}}= 0,  \quad \forall \alpha  \in B.
\end{align}
Set
\begin{align}\label{4.18}
\phi  = \sigma _{l + 1} (D_{x,t}^2 u) ,
\end{align}
In the following, we will prove a differential inequality
\begin{align}\label{4.19}
\sum_{ij=1}^n F^{ij} \phi_{ij}  - \phi _t \le  C(\phi  + |\nabla _x \phi |) \quad \text{ in } \mathcal {O}\times (t_0-\delta,
t_0].
\end{align}
Then by the strong maximum principle and the method of continuity, we can prove Theorem \ref{th1.3} under CASE 2.

In the $y$ coordinates, we have from \eqref{4.17}
\begin{align*}
\phi  = \sigma _{l + 1} (D_{x,t}^2 u) = \sigma _{l + 1} (D_y^2 u) = \sigma _l (G)u_{y_{n + 1}y_{n + 1}}  \ge 0,
\end{align*}
so we have
 \begin{align}\label{4.20}
u_{y_{n + 1}y_{n + 1}}  = O(\phi ).
  \end{align}
Taking the first derivatives of $\phi$ in $x$, we have
\begin{align}
\phi _i  =& \frac{{\partial \phi }}{{\partial x_i }} = \sum\limits_{\alpha  = 1}^{n + 1} {\sigma _l (D_y^2 u|\alpha )u_{\alpha \alpha i} } \notag \\
=& \sum\limits_{\alpha  \in G} {\sigma _l (D_y^2 u|\alpha )u_{\alpha \alpha i} }  + \sum\limits_{\alpha  \in B} {\sigma _l (D_y^2 u|\alpha )u_{\alpha \alpha i} }  + \sum\limits_{\alpha  = n + 1} {\sigma _l (D_y^2 u|\alpha )u_{\alpha \alpha i} }  \notag\\
=& \sigma _l (G)u_{y_{n + 1}y_{n + 1} x_i}  + O(\phi ), \notag
\end{align}
so
\begin{align}\label{4.21}
u_{y_{n + 1}y_{n + 1} x_i} = O(\phi +|\nabla_x \phi|),
\end{align}
Similarly, taking the first derivatives of $\phi$ in $t$, we have
\begin{align}\label{4.22}
\phi _t  = \frac{{\partial \phi }}{{\partial t}} = \sum\limits_{\alpha  = 1}^{n + 1} {\sigma _l (D_y^2 u|\alpha )u_{\alpha \alpha t} }  = \sigma _l (G)u_{y_{n + 1}y_{n + 1}t}  + O(\phi )
\end{align}
Taking the second derivatives of $\phi$ in $y$ coordinates, we have
\begin{align}\label{4.23}
\phi _{\alpha \beta }  =& \frac{{\partial ^2 \phi }}{{\partial y_\alpha  \partial y_\beta  }} \notag \\
=& \sum\limits_{\gamma  = 1}^{n + 1} {\frac{{\partial \sigma _{l + 1} (D_y^2 u)}}{{\partial u_{\gamma \gamma } }}u_{\gamma \gamma \alpha \beta } }  + \sum\limits_{\gamma  \ne \eta } {\frac{{\partial ^2 \sigma _{l + 1} }}{{\partial u_{\gamma \gamma } \partial u_{\eta \eta } }}u_{\gamma \gamma \alpha } u_{\eta \eta \beta } }  + \sum\limits_{\gamma  \ne \eta } {\frac{{\partial ^2 \sigma _{l + 1} }}{{\partial u_{\gamma \eta } \partial u_{\eta \gamma } }}u_{\gamma \eta \alpha } u_{\eta \gamma \beta } } \notag \\
=& \sum\limits_{\gamma  = 1}^{n + 1} {\sigma _l (D_y^2 u|\gamma )u_{\gamma \gamma \alpha \beta } }  + \sum\limits_{\gamma  \ne \eta } {\sigma _{l - 1} (D_y^2 u|\gamma \eta )u_{\gamma \gamma \alpha } u_{\eta \eta \beta } }  - \sum\limits_{\gamma  \ne \eta } {\sigma _{l - 1} (D_y^2 u|\gamma \eta )u_{\gamma \eta \alpha } u_{\eta \gamma \beta } }
\end{align}
where
\begin{align}\label{4.24}
\sum\limits_{\gamma  = 1}^{n + 1} {\sigma _l (D_y^2 u|\gamma )u_{\gamma \gamma \alpha \beta } }  =& \sum\limits_{\gamma  \in G} {\sigma _l (D_y^2 u|\gamma )u_{\gamma \gamma \alpha \beta } }  + \sum\limits_{\gamma  = n + 1} {\sigma _l (D_y^2 u|\gamma )u_{\gamma \gamma \alpha \beta } }  \notag\\
=& \sigma _l (G)u_{n + 1n + 1\alpha \beta }  + O(\phi ),
\end{align}
\begin{align}\label{4.25}
\sum\limits_{\gamma  \ne \eta } {\sigma _{l - 1} (D_y^2 u|\gamma \eta )u_{\gamma \gamma \alpha } u_{\eta \eta \beta } }  =& \sum\limits_{\scriptstyle \gamma \eta  \in G \hfill \atop
\scriptstyle \gamma  \ne \eta  \hfill} {\sigma _{l - 1} (D_y^2 u|\gamma \eta )u_{\gamma \gamma \alpha } u_{\eta \eta \beta } }  + \sum\limits_{\scriptstyle \gamma  = n + 1 \hfill \atop
\scriptstyle \eta  \in G \hfill} {\sigma _{l - 1} (D_y^2 u|\gamma \eta )u_{\gamma \gamma \alpha } u_{\eta \eta \beta } }  \notag\\
&+ \sum\limits_{\scriptstyle \gamma  \in G \hfill \atop
\scriptstyle \eta  = n + 1 \hfill} {\sigma _{l - 1} (D_y^2 u|\gamma \eta )u_{\gamma \gamma \alpha } u_{\eta \eta \beta } }  \notag\\
=& O(\phi ) + \sum\limits_{\eta  \in G} {\sigma _{l - 1} (G|\eta )u_{\eta \eta \beta } } u_{n + 1n + 1\alpha }  + \sum\limits_{\gamma  \in G} {\sigma _{l - 1} (G|\gamma )u_{\gamma \gamma \alpha } } u_{n + 1n + 1\beta }  \notag\\
=& \sigma _l (G)[\sum\limits_{\eta  \in G} {\frac{{u_{\eta \eta \beta } }}{{u_{\eta \eta } }}} u_{n + 1n + 1\alpha }  + \sum\limits_{\gamma  \in G}{\frac{{u_{\gamma \gamma \alpha } }}{{u_{\gamma \gamma } }}} u_{n + 1n + 1\beta } ] + O(\phi ),
\end{align}
and
\begin{align}\label{4.26}
\sum\limits_{\gamma  \ne \eta } {\sigma _{l - 1} (D_y^2 u|\gamma \eta )u_{\gamma \eta \alpha } u_{\eta \gamma \beta } }  =& \sum\limits_{\scriptstyle \gamma \eta  \in G \hfill \atop
\scriptstyle \gamma  \ne \eta  \hfill} {\sigma _{l - 1} (D_y^2 u|\gamma \eta )u_{\gamma \eta \alpha } u_{\eta \gamma \beta } }  + \sum\limits_{\scriptstyle \gamma  = n + 1 \hfill \atop
\scriptstyle \eta  \in G \hfill} {\sigma _{l - 1} (D_y^2 u|\gamma \eta )u_{\gamma \eta \alpha } u_{\eta \gamma \beta } } \notag \\
&+ \sum\limits_{\scriptstyle \gamma  \in G \hfill \atop
\scriptstyle \eta  = n + 1 \hfill} {\sigma _{l - 1} (D_y^2 u|\gamma \eta )u_{\gamma \eta \alpha } u_{\eta \gamma \beta } } \notag \\
=& O(\phi ) + \sum\limits_{\eta  \in G} {\sigma _{l - 1} (G|\eta )u_{n + 1\eta \alpha } u_{\eta n + 1\beta } }  + \sum\limits_{\gamma  \in G} {\sigma _{l - 1} (G|\gamma )u_{\gamma n + 1\alpha } u_{n + 1\gamma \beta } } \notag \\
=& 2\sigma _l (G)\sum\limits_{\eta  \in G} {\frac{{u_{n + 1\eta \alpha } u_{\eta n + 1\beta } }}{{u_{\eta \eta } }}}  + O(\phi ).
\end{align}
So we have
\begin{align}\label{4.27}
\phi _{\alpha \beta }  =& \sigma _l (G)u_{n + 1n + 1\alpha \beta }  - 2\sigma _l (G)\sum\limits_{\eta  \in G} {\frac{{u_{n + 1\eta \alpha } u_{\eta n + 1\beta } }}{{u_{\eta \eta } }}}  \notag \\
&+ \sigma _l (G)[\sum\limits_{\eta  \in G} {\frac{{u_{\eta \eta \beta } }}{{u_{\eta \eta } }}} u_{n + 1n + 1\alpha }  + \sum\limits_{\gamma  \in G} {\frac{{u_{\gamma \gamma \alpha } }}{{u_{\gamma \gamma } }}} u_{n + 1n + 1\beta } ] + O(\phi ).
\end{align}
Then
\begin{align}\label{4.28}
\sum\limits_{ij = 1}^n {F^{ij} \phi _{ij} }  =& \sum\limits_{ij = 1}^n {F^{ij} \sum\limits_{\alpha \beta  = 1}^{n + 1} {P_{i\alpha } P_{j\beta } \phi _{\alpha \beta } } } \notag \\
=& \sigma _l (G)\sum\limits_{ij = 1}^n {F^{ij} \sum\limits_{\alpha \beta  = 1}^{n + 1} {P_{i\alpha } P_{j\beta } } u_{n + 1n + 1\alpha \beta } }  - 2\sigma _l (G)\sum\limits_{ij = 1}^n {F^{ij} \sum\limits_{\eta  \in G} {\frac{{[\sum\limits_{\alpha  = 1}^{n + 1} {P_{i\alpha } u_{n + 1\eta \alpha } } ][\sum\limits_{\beta  = 1}^{n + 1} {P_{j\beta } u_{\eta n + 1\beta } } ]}}{{u_{\eta \eta } }}} }  \notag\\
&+ \sigma _l (G)\sum\limits_{ij = 1}^n {F^{ij} [\sum\limits_{\eta  \in G} {\frac{{\sum\limits_{\beta  = 1}^{n + 1} {P_{j\beta } } u_{\eta \eta \beta } }}{{u_{\eta \eta } }}} \sum\limits_{\alpha  = 1}^{n + 1} {P_{i\alpha } u_{n + 1n + 1\alpha } }  + \sum\limits_{\gamma  \in G} {\frac{{\sum\limits_{\alpha  = 1}^{n + 1} {P_{i\alpha } } u_{\gamma \gamma \alpha } }}{{u_{\gamma \gamma } }}} \sum\limits_{\beta  = 1}^{n + 1} {P_{j\beta } u_{n + 1n + 1\beta } } ]}  + O(\phi ) \notag\\
=& \sigma _l (G)\sum\limits_{ij = 1}^n {F^{ij} u_{n + 1n + 1ij} }  - 2\sigma _l (G)\sum\limits_{ij = 1}^n {F^{ij} \sum\limits_{\eta  \in G} {\frac{{u_{n + 1\eta i} u_{\eta n + 1j} }}{{u_{\eta \eta } }}} }  \notag\\
&+ \sigma _l (G)\sum\limits_{ij = 1}^n {F^{ij} [\sum\limits_{\eta  \in G} {\frac{{u_{\eta \eta j} }}{{u_{\eta \eta } }}} u_{n + 1n + 1i}  + \sum\limits_{\gamma  \in G} {\frac{{u_{\gamma \gamma i} }}{{u_{\gamma \gamma } }}} u_{n + 1n + 1j} ]}  + O(\phi )
\end{align}
By \eqref{4.21}, we have
\begin{align}\label{4.29}
 \sum\limits_{ij = 1}^n {F^{ij} \phi _{ij} }  = \sigma _l (G)\left[ {\sum\limits_{ij = 1}^n {F^{ij} u_{n + 1n + 1ij} }  - 2\sum\limits_{\eta  \in G} {\sum\limits_{ij = 1}^n {F^{ij} } \frac{{u_{n + 1\eta i} u_{\eta n + 1j} }}{{u_{\eta \eta } }}} } \right] + O(\phi  + |\nabla _x \phi |)
 \end{align}
From \eqref{4.22} and \eqref{4.29}, we have
\begin{align}\label{4.30}
 \sum\limits_{ij = 1}^n {F^{ij} \phi _{ij} }  - \phi _t  =& \sigma _l (G)\big[ {( {\sum\limits_{ij = 1}^n {F^{ij} u_{n + 1n + 1ij} }  - u_{n + 1n + 1t} } ) - 2\sum\limits_{\eta  \in G} {\sum\limits_{ij = 1}^n {F^{ij} } \frac{{u_{n + 1\eta i} u_{\eta n + 1j} }}{{u_{\eta \eta } }}} } \big] \notag\\
 &+ O(\phi  + |\nabla _x \phi |)
 \end{align}

For the first term in the right hand side of \eqref{4.30}, we use the equation
\begin{align*}
 u_t  = F(\nabla ^2 u,\nabla u,u,x,t)
 \end{align*}
Taking the second derivative in $y_{n + 1}$, we have
\begin{align}\label{4.31}
u_{n + 1n + 1 t}  =& \sum\limits_{ij = 1}^n {F^{ij} } u_{n + 1 n + 1 ij}  + \sum\limits_{i = 1}^n {F^{u_i } } u_{n + 1 n + 1 i}  + F^u u_{n + 1 n + 1 }  \notag\\
&+ \sum\limits_{ijkl = 1}^n {F^{ij,kl} u_{ij\gamma } u_{kl\gamma } }  + 2\sum\limits_{ijk = 1}^n {F^{ij,u_k } } u_{ij\gamma } u_{k\gamma }  + 2\sum\limits_{ij = 1}^n {F^{ij,u} } u_{ijn + 1 } u_{n + 1} \notag\\
& + 2\sum\limits_{ijk = 1}^n {F^{ij,x_k } } u_{ijn + 1 } \frac{{\partial x_k }}{{\partial y_{n + 1}  }} + 2\sum\limits_{ij = 1}^n {F^{ij,t} } u_{ijn + 1} \frac{{\partial t}}{{\partial y_{n + 1} }}+ \sum\limits_{ij = 1}^n {F^{u_i ,u_j } } u_{in + 1} u_{jn + 1 } \notag\\
& + 2\sum\limits_{i = 1}^n {F^{u_i ,u} } u_{in + 1 } u_{n + 1}    + 2\sum\limits_{ik = 1}^n {F^{u_i ,x_k } } u_{in + 1 } \frac{{\partial x_k }}{{\partial y_{n + 1}  }}+ 2\sum\limits_{i = 1}^n {F^{u_i ,t} } u_{in + 1 } \frac{{\partial t}}{{\partial y_{n + 1}  }} \notag\\
&+ F^{u,u} u_{n + 1} ^2   + 2\sum\limits_{k = 1}^n {F^{u,x_k } } u_{n + 1} \frac{{\partial x_k }}{{\partial y_{n + 1}  }} + 2F^{u,t} u_{n + 1}  \frac{{\partial t}}{{\partial y_{n + 1}  }} \notag\\
&+ \sum\limits_{ik = 1}^n {F^{x_i ,x_k } } \frac{{\partial x_i }}{{\partial y_{n + 1} }}\frac{{\partial x_k }}{{\partial y_{n + 1} }}+ 2\sum\limits_{i = 1}^n {F^{x_i ,t} } \frac{{\partial x_i }}{{\partial y_{n + 1}  }}\frac{{\partial t}}{{\partial y_{n + 1} }} + F^{t,t} \left( {\frac{{\partial t}}{{\partial y_{n + 1} }}} \right)^2
\end{align}
From \eqref{4.17}, \eqref{4.20} and \eqref{4.21}, we have
\begin{align}
\label{4.32}&u_{n + 1 n + 1 i}  = O(\phi  + |\nabla _x \phi |), \forall i=1, \cdots, n;\\
\label{4.33}&u_{\alpha \alpha }  =0, \forall \alpha  \in B;  u_{n + 1 n + 1 }  = O(\phi ),\\
\label{4.34}&u_{in + 1 }  = \frac{{\partial u_{n + 1}  }}{{\partial x_i }} = \sum\limits_{\eta  = 1}^{n + 1} {\frac{{\partial u_{n + 1}   }}{{\partial y_\eta  }}\frac{{\partial y_\eta  }}{{\partial x_i }}}  = \sum\limits_{\eta  = 1}^{n + 1} {u_{n + 1  \eta } P_{i\eta } }  = u_{n + 1 n + 1 } P_{in + 1 }  = O(\phi ).
\end{align}
And from \eqref{4.17} and Lemma \ref{lem2.6}, we have for $i \text{ or } j \in B$
\begin{align}
|u_{ijn + 1  } | \leq C ( u_{ii} u_{jj})^{\frac{1}{4}} =  0, \notag
\end{align}
so we have
\begin{align}\label{4.35}
u_{ijn + 1  }  =  0.
\end{align}
Then
\begin{align}\label{4.36}
&u_{n + 1 n + 1 t}  - \sum\limits_{ij = 1}^n {F^{ij} } u_{n + 1 n + 1 ij}  \notag \\
=& \sum\limits_{ijkl = 1}^n {F^{ij,kl} u_{ijn + 1} u_{kln + 1} }  + 2\sum\limits_{ij = 1}^n {F^{ij,u} } u_{ijn + 1 } u_{n + 1} + 2\sum\limits_{ijk = 1}^n {F^{ij,x_k } } u_{ijn + 1} \frac{{\partial x_k }}{{\partial y_{n + 1}  }}\notag\\
& + 2\sum\limits_{ij = 1}^n {F^{ij,t} } u_{ijn + 1} \frac{{\partial t}}{{\partial y_{n + 1} }} + F^{u,u} u_{n + 1}^2  + 2\sum\limits_{k = 1}^n {F^{u,x_k } } u_{n + 1} \frac{{\partial x_k }}{{\partial y_{n + 1} }} + 2F^{u,t} u_{n + 1}  \frac{{\partial t}}{{\partial y_{n + 1} }} \notag \\
&+ \sum\limits_{ik = 1}^n {F^{x_i ,x_k } } \frac{{\partial x_i }}{{\partial y_{n + 1} }}\frac{{\partial x_k }}{{\partial y_{n + 1}}} + 2\sum\limits_{i = 1}^n {F^{x_i ,t} } \frac{{\partial x_i }}{{\partial y_{n + 1}  }}\frac{{\partial t}}{{\partial y_{n + 1} }} + F^{t,t} \left( {\frac{{\partial t}}{{\partial y_{n + 1} }}} \right)^2  + O(\phi )\notag\\
=& \sum\limits_{ijkl \in G} {F^{ij,kl} u_{ijn + 1 } u_{kln + 1} }  + 2\sum\limits_{ij \in G} {F^{ij,u} u_{ijn + 1 } u_{n + 1}  }  + 2\sum\limits_{ij \in G} {F^{ij,x_k } u_{ijn + 1 } \frac{{\partial x_k }}{{\partial y_{n + 1}  }}}  \notag\\
&+ 2\sum\limits_{ij \in G} {F^{ij,t} u_{ijn + 1 } \frac{{\partial t}}{{\partial y_{n + 1} }}}  + F^{u,u} u_{n + 1}  ^2  + 2\sum\limits_{k = 1}^n {F^{u,x_k } } u_{n + 1} \frac{{\partial x_k }}{{\partial y_{n + 1} }} + 2F^{u,t} u_{n + 1}  \frac{{\partial t}}{{\partial y_{n + 1} }} \notag \\
&+ \sum\limits_{ik = 1}^n {F^{x_i ,x_k } } \frac{{\partial x_i }}{{\partial y_{n + 1} }}\frac{{\partial x_k }}{{\partial y_{n + 1} }} + 2\sum\limits_{i = 1}^n {F^{x_i ,t} } \frac{{\partial x_i }}{{\partial y_{n + 1} }}\frac{{\partial t}}{{\partial y_{n + 1} }} + F^{t,t} \left( {\frac{{\partial t}}{{\partial y_{n + 1}  }}} \right)^2+O(\phi  + |\nabla _x \phi |).
\end{align}

For the second part in the right hand side of \eqref{4.30}, we have the following CLAIM:
\begin{align}\label{4.37}
 \sum\limits_{\eta  \in G} {\sum\limits_{ij = 1}^n {F^{ij} } \frac{{u_{n + 1\eta i} u_{\eta n + 1j} }}{{u_{\eta \eta } }}}  \ge \sum\limits_{kl \in G} {\sum\limits_{ij \in G} {F^{ij} u_{n + 1ki} u_{n + 1lj} u^{kl} } }+ O(\phi  + |\nabla _x \phi |)
\end{align}

If the CLAIM holds, denote
\begin{align}\label{4.38}
Q=& \sum\limits_{ijkl \in G} {F^{ij,kl} u_{ijn + 1} u_{kln + 1} } +\sum\limits_{kl \in G} {\sum\limits_{ij \in G} {F^{ij} u_{n + 1ki} u_{n + 1lj} u^{kl} } }  \notag\\
&+ 2\sum\limits_{ij \in G} {F^{ij,u} u_{ijn + 1 } u_{n + 1}  }  + 2\sum\limits_{ij \in G} {F^{ij,x_k } u_{ijn + 1 } \frac{{\partial x_k }}{{\partial y_{n + 1}  }}}  + 2\sum\limits_{ij \in G} {F^{ij,t} u_{ijn + 1 } \frac{{\partial t}}{{\partial y_{n + 1} }}} \notag\\
& + F^{u,u} u_{n + 1}  ^2  + 2\sum\limits_{k = 1}^n {F^{u,x_k } } u_{n + 1} \frac{{\partial x_k }}{{\partial y_{n + 1} }} + 2F^{u,t} u_{n + 1}  \frac{{\partial t}}{{\partial y_{n + 1} }} \notag \\
&+ \sum\limits_{ik = 1}^n {F^{x_i ,x_k } } \frac{{\partial x_i }}{{\partial y_{n + 1} }}\frac{{\partial x_k }}{{\partial y_{n + 1} }} + 2\sum\limits_{i = 1}^n {F^{x_i ,t} } \frac{{\partial x_i }}{{\partial y_{n + 1} }}\frac{{\partial t}}{{\partial y_{n + 1} }} + F^{t,t} \left( {\frac{{\partial t}}{{\partial y_{n + 1}  }}} \right)^2.
\end{align}
By the structural condition \eqref{1.7} ( that is Lemma \ref{lem2.5}), we have
\begin{align}\label{4.39}
Q \geq 0.
\end{align}
Then by \eqref{4.30}, \eqref{4.36}, \eqref{4.37} and \eqref{4.39}, we have
\begin{align}\label{4.40}
 \sum\limits_{ij = 1}^n {F^{ij} \phi _{ij} }  - \phi _t \leq & -\sigma _l (G) Q+ O(\phi  + |\nabla _x \phi |) \leq C(\phi  + |\nabla _x \phi |).
 \end{align}
So \eqref{4.19} holds, and Theorem \ref{th1.3} holds under CASE 2.
\qed

\subsection{ Proof of the CLAIM \eqref{4.37}}
Now we give the proof of the CLAIM \eqref{4.37} as follows.

First, we consider a special case: $ F^{ij}= \delta_{ij}$.
That is, we need to prove
\begin{align}\label{4.41}
 \sum\limits_{\eta  \in G} {\sum\limits_{i = 1}^n  \frac{{u_{n + 1\eta i} ^2}}{{u_{\eta \eta } }}}  \ge \sum\limits_{kl \in G} {\sum\limits_{i \in G} { u_{n + 1ki} u_{n + 1li} u^{kl} } }+ O(\phi  + |\nabla _x \phi |).
\end{align}
Form \eqref{4.35} and \eqref{4.21}, we have
\begin{align*}
& u_{n + 1\eta i}  = 0, \eta  \in B \text{ or } i \in B, \\
& u_{n + 1n + 1i}  = O(\phi  + |\nabla _x \phi |).
\end{align*}
Since $D_y^2 u$ is diagonal, by the approximation, we have for $i \in G$
\begin{align}
 \sum\limits_{\eta  \in G} {\frac{{u_{n + 1\eta i} u_{\eta n + 1i} }}{{u_{\eta \eta } }}}  = \mathop {\lim }\limits_{\varepsilon  \to 0 + } (D_y u_{n + 1i} )\left( {D_y^2 u + \varepsilon I} \right)^{ - 1} (D_y u_{n + 1i} )^T  + O(\phi  + |\nabla _x \phi |),
 \end{align}
where
\begin{align}
(D_y u_{n + 1i} )\left( {D_y^2 u + \varepsilon I} \right)^{ - 1} (D_y u_{n + 1i} )^T  =& (D_{x,t} u_{n + 1i} )P^T \left( {D_y^2 u + \varepsilon I} \right)^{ - 1} P(D_{x,t} u_{n + 1i} )^T  \notag \\
=& (D_{x,t} u_{n + 1i} )\left( {D_{x,t}^2 u + \varepsilon I} \right)^{ - 1} (D_{x,t} u_{n + 1i} )^T.
\end{align}

Denote
\begin{align}
 C := u_{tt}  + \varepsilon  - \sum\limits_{i = 1}^l {\frac{{u_{x_i t} ^2 }}{{u_{x_i x_i }  + \varepsilon }}}  > 0,
 \end{align}
then
\begin{align}
\left( {D_{x,t}^2 u + \varepsilon I} \right)^{ - 1}  =& diag(\frac{1}{{u_{x_1 x_1 }  + \varepsilon }}, \cdots ,\frac{1}{{u_{x_l x_l }  + \varepsilon }},\frac{1}{\varepsilon }, \cdots ,\frac{1}{\varepsilon },0)  \notag \\
&+ \frac{1}{C}( - \frac{{u_{x_1 t} }}{{u_{x_1 x_1 }  + \varepsilon }}, \cdots , - \frac{{u_{x_l t} }}{{u_{x_l x_l }  + \varepsilon }},0, \cdots ,0,1)^T ( - \frac{{u_{x_1 t} }}{{u_{x_1 x_1 }  + \varepsilon }}, \cdots , - \frac{{u_{x_l t} }}{{u_{x_l x_l }  + \varepsilon }},0, \cdots ,0,1)  \notag \\
\ge& diag(\frac{1}{{u_{x_1 x_1 }  + \varepsilon }}, \cdots ,\frac{1}{{u_{x_l x_l }  + \varepsilon }},0, \cdots ,0,0). \notag
\end{align}
So
\begin{align}
 (D_y u_{n + 1i} )\left( {D_y^2 u + \varepsilon I} \right)^{ - 1} (D_y u_{n + 1i} )^T  \ge \sum\limits_{k \in G} {\frac{{u_{n + 1ki} u_{n + 1ki} }}{{u_{x_k x_k }  + \varepsilon }}}.
 \end{align}
Then we have for $i \in G$
\begin{align}
 \sum\limits_{\eta  \in G} {\frac{{u_{n + 1\eta i} u_{\eta n + 1i} }}{{u_{\eta \eta } }}}  \ge& \mathop {\lim }\limits_{\varepsilon  \to 0 + } \sum\limits_{k \in G} {\frac{{u_{n + 1ki} u_{n + 1ki} }}{{u_{x_k x_k }  + \varepsilon }}}  + O(\phi  + |\nabla _x \phi |)  \notag  \\
  =& \sum\limits_{k \in G} {\frac{{u_{n + 1ki} u_{n + 1ki} }}{{u_{x_k x_k } }}}  + O(\phi  + |\nabla _x \phi |)  \notag \\
  =& \sum\limits_{kl \in G} {u_{n + 1ki} u_{n + 1li} u^{kl} }  + O(\phi  + |\nabla _x \phi |).
 \end{align}
Hence, \eqref{4.41} holds.

For the general case, the CLAIM also holds following the above proof.

\section{Discussions}

In fact, there are many equations satisfying the conditions \eqref{1.7}.

\begin{proposition} \label{prop5.1}
(1) All the linear operators satisfy conditions \eqref{1.7}.

(2)the Hessian operators $\sigma_k^{\frac{1}{k}}$ and $(\frac{\sigma_k}{\sigma_l})^{\frac{1}{k-l}}$ ($k>l>0$) satisfy the condition \eqref{1.7} for the convex admissible solutions ( that is $D^2u \geq 0$, and $D^2 u \in \Gamma_k$ on $\Omega \times (0, T]$).

(3) If $g$ is a non-decreasing and convex function and $F_1$, $\cdots$,
$F_m$ satisfy condition \eqref{1.7}, then $F = g(F_1,\cdots, F_m)$
also satisfies condition \eqref{1.7}. In particular,  if $F_1$ and
$F_2$ are in the class, so are $F_1 +F_2$ and $F_1^\alpha$( where
$F_1 > 0$) for any $\alpha \geqslant 1$.
\end{proposition}
Through a direct calculation and using \eqref{2.7}, we can get the proof. Also we can find the proof of Proposition \ref{prop5.1} easily from \cite{BG09, BG10, CH13}.

\textbf{Acknowledgement}. The author would like to express sincere
gratitude to Prof. Xi-Nan Ma for the constant encouragement and helpful
suggestions in this subject. Also, the author would like to thank Prof. Pengfei Guan for the advice on the choice of coordinates in Dec. 2012.

\end{document}